\documentclass[11pt]{amsart}
\usepackage{amsmath}
\usepackage{amssymb}
\usepackage{amsthm}
\usepackage{amsfonts}
\usepackage{latexsym}
\usepackage{graphicx}
\usepackage{hyperref}
\usepackage{enumerate}
\usepackage[all]{xy}
\usepackage{chngcntr}
\usepackage{xcolor}
\usepackage{tikz-cd}
\usepackage{tikz}
\usepackage{multicol}
\usepackage{adjustbox}

\newtheorem{thm}{Theorem}[section]

\newtheorem{lem}[thm]{Lemma}
\newtheorem{prop}[thm]{Proposition}

\theoremstyle{definition}

\newcommand{\Z}{\mathbb Z}

\newcommand{\R}{\mathbb R}
\newcommand{\C}{\mathbb C}

\newcommand{\mf}{\mathfrak}
\newcommand{\mc}{\mathcal}
\newcommand{\mb}{\mathbf}
\newcommand{\mh}{\mathbb}

\newcommand{\mr}{\mathrm}

\newcommand{\enuma}[1]{\begin{enumerate}[\textup{(}a\textup{)}] {#1} \end{enumerate}}
\newcommand{\Fr}{\mathrm{Frob}}
\newcommand{\Sc}{\mathrm{sc}}

\newcommand{\nr}{\mathrm{nr}}

\newcommand{\Res}{\mathrm{Res}}

\newcommand{\der}{\mathrm{der}}

\newcommand{\Hom}{\mathrm{Hom}}

\newcommand{\isom}{\xrightarrow{\,\sim\,}}

\newcommand{\ff}{\mathfrak{f}}

\newcommand{\Xo}{\mathfrak{X}^0}

\numberwithin{equation}{section}

\begin{document}

\title{On depth-zero characters of $p$-adic groups}
\author{Maarten Solleveld}
\address{Institute for Mathematics, Astrophysics and Particle Physics, 
Radboud Universiteit Nijmegen, Nijmegen,
The Netherlands}
\email{m.solleveld@science.ru.nl}

\author{Yujie Xu}
\address{Department of Mathematics, 
Columbia University, 
New York, NY, USA}
\email{xu.yujie@columbia.edu}

\date{\today}

\begin{abstract}
We show new properties of the Langlands correspondence for arbitrary 
tori over local fields. 
Furthermore, we give a detailed analysis of depth-zero characters of reductive $p$-adic groups, for groups that may be \textit{wildly ramified}. We present several different definitions
of ``depth-zero'' for characters, and show that these notions are in fact equivalent.
These results are useful for proving new cases of local Langlands correspondences, in particular for depth zero representations. 
\end{abstract}

\maketitle

\section{Introduction}

The local Langlands correspondence
(LLC) for tori is a starting point for the entire local
Langlands program. Langlands famously discovered it in the
1960s and published it much later in \cite{Lan2}. 
This correspondence is crucial to parametrizing those irreducible 
representations of reductive groups (over local fields) whose construction
involves characters of maximal tori -- these constitute the majority of all
irreducible representations of such groups. 

Let $F$ be a local field with Weil group $\mb W_F$. For an $F$-torus $\mc T$, 
let $T^\vee$ be its complex dual group, endowed with a $\mb W_F$-action. 
Recall the following:

\begin{thm}[LLC for tori]\label{thm:1.6} 
\textup{\cite[\S 9]{Bor}, \cite[Theorem 2]{Lan2}, \cite[Theorem 7.5]{Yu}} \\
There exists a natural isomorphism of topological groups
\begin{equation}\label{eqn:LLC-for-tori}
\Hom (\mc T (F), \C^\times) \isom H^1 (\mb W_F, T^\vee) :\:
\chi \mapsto \varphi_\chi .
\end{equation}
The family of these isomorphisms, for all tori over local fields $F$,
is functorial with respect to homomorphisms of $F$-tori
and generalizes local class field theory.
\end{thm}

The topologies in Theorem \ref{thm:1.6} are the compact-open topologies for maps 
$\mc T (F) \to \C^\times$ and for maps $\mb W_F \to T^\vee$. While working on 
\cite{SoXu1, SoXu2}, we encountered some interesting properties 
of the LLC for tori, which were not in previous literature.

Consider a separable extension $E/F$ of local fields. 
It is well-known that Artin reciprocity transfers the norm map 
$N_{E/F} : E^\times \to F^\times$ into the inclusion $\mb W_E \to \mb W_F$
\cite[Proposition 5.4]{Neu}. 
We show that the LLC for tori \eqref{eqn:LLC-for-tori} enjoys a similar property. 
To formulate this precisely, we need the norm map for an $F$-torus $\mc T$, 
which can be defined as
\[
N_{E/F} : 
\mc T (E)  
\to  
\mc T (F),\quad 
t 
\mapsto  
\prod\nolimits_{\gamma \in \mb W_F / \mb W_E} \gamma (t).
\]
Let ${N^*_{E/F}} : \Hom (\mc T(F),\C^\times) \to \Hom (\mc T (E),\C^\times)$ denote
composition by $N_{E/F}$.

\begin{prop}[see Proposition \ref{prop:1.8}]\label{intro:prop:1.8} \ \\
The following diagram commutes:
\[\begin{tikzcd}
    \Hom (\mc T (E), \C^\times) \arrow[]{r}{\sim} & H^1 (\mb W_E ,T^\vee) \\
\Hom (\mc T (F), \C^\times) \arrow[]{r}{\sim}\arrow[]{u}[swap]{N^*_{E/F}} & 
H^1 (\mb W_F, T^\vee)\arrow[]{u}{\Res^{\mb W_F}_{\mb W_E}}
\end{tikzcd}\]
In other words, $\varphi_\chi |_{\mb W_E} = \varphi_{\chi \circ N_{E/F}}$
for all smooth characters $\chi : \mc T (F) \to \C^\times$.
\end{prop}

\noindent
The main arithmetic significance of Theorem \ref{thm:1.6} for non-archimedean 
local fields $F$ lies in the fact that \eqref{eqn:LLC-for-tori} preserve depths 
\cite[Theorem 7.10]{Yu}, in the following sense. If $\chi$ is trivial on the $r$-th 
filtration subgroup $T_r$ of $T = \mc T (F)$, then $\varphi_\chi$ is trivial on the 
$r$-th ramification subgroup $\mb W_F^r \subset \mb W_F$, and vice versa. However, 
in \cite{Yu}, this is only proven under the assumption that $\mc T$ splits over a 
tamely ramified extension of $F$. In fact, it is known that Theorem \ref{thm:1.6} 
often does not preserve depths for wildly ramified tori (see for example 
\cite[\S 7]{MiPa}).

Fortunately, things become simpler in depth zero. Let $T^0$ be the unique 
parahoric subgroup of $T = \mc T (F)$, and let $T_{0+}$ be its pro-$p$ radical.
Recall that a character of $T$ has depth zero if it is trivial on $T_{0+}$. Let 
\[
\mb P_F = \mb W_F^{0+} \;\subset\; \mb I_F = \mb W_F^0 \;\subset\; \mb W_F
\]
be the wild inertia and inertia subgroups of $\mb W_F$.~Recall that a Langlands 
parameter $\varphi \in H^1 (\mb W_F, T^\vee)$, presented as a 1-cocycle 
$\mb W_F \to T^\vee$, has depth zero if $\varphi (w) = 1$ for all $w \in \mb P_F$. 
The condition $\varphi |_{\mb P_F} = 1$ guarantees that $\varphi (\mb W_F) \subset 
T^{\vee,\mb P_F}$, so the depth-zero part of $H^1 (\mb W_F, T^\vee)$ is 
$H^1 \big( \mb W_F / \mb P_F, T^{\vee,\mb P_F} \big)$.
The following result is new for $F$-tori that only split over a wildly ramified extension of $F$. 

\begin{prop}[see Proposition \ref{prop:1.7}]\label{intro:prop:1.7} \ \\
Let $\mc T$ be an arbitrary torus over a non-archimedean local field $F$. 
The local Langlands correspondence for tori restricts to an isomorphism
between the depth-zero parts on both sides of the correspondence:
\[
\Hom (T / T_{0+}, \C^\times) \cong H^1 \big( \mb W_F / \mb P_F, T^{\vee,\mb P_F} \big).
\]
\end{prop}

For comparison, recall that the group of weakly unramified characters of $T$ is 
$\Hom (T/T^0, \C^\times)$. By \cite[\S 3.3.1]{Hai}, the LLC for tori restricts to an isomorphism
\begin{equation}\label{eq:1.15}
\Hom (T / T^0 ,\C^\times) \cong H^1 \big( \mb W_F / \mb I_F, T^{\vee,\mb I_F} \big) . 
\end{equation}

Our next results concern characters of arbitrary reductive groups $G = \mc G (F)$,
over a local field $F$.
In \cite{Lan2,Bor}, Langlands also parametrized these, as follows. 
Let $Z(G^\vee)$ be the centre of the complex dual group $G^\vee$ and let 
$G / G_\Sc$ be the cokernel of the canonical map $G_\Sc \to G$, where $G_\Sc$ denotes 
the simply connected cover of the derived group of $G$. In these terms, Langlands 
constructed a natural homomorphism
\begin{equation}\label{eq:1.1}
H^1 (\mb W_F, Z(G^\vee)) \longrightarrow \Hom ( G/G_\Sc, \C^\times) .
\end{equation}
We prove in Theorem \ref{thm:1.1} that this map is bijective.
Note that bijectivity was neither claimed nor proven in \cite{Lan2,Bor}. 

Suppose now that $F$ is non-archimedean, so that we can consider depths.
In order to study depth-zero representations of $G$, it is tremendously helpful to know 
which characters of $G / G_\Sc$ have depth zero (see \cite{SoXu1}). There are several 
conceivable and reasonable notions 
of depth zero, depending on the point of view. If one views \eqref{eq:1.1} as a 
generalization of the LLC for tori, then it is natural to require that $\chi : G/G_\Sc \to
\C^\times$ has depth zero when restricted to every (or one) maximal torus $T \subset G$.
On the other hand, considered as a $G$-representation, one would require that $\chi$ has 
depth zero when restricted to one (or every) parahoric subgroup $G_{\ff,0}$ (i.e.~$\chi$ is
trivial on the pro-unipotent radical $G_{\ff,0+}$ of $G_{\ff,0}$). Fortunately, it
turns out that all these notions are equivalent.

\begin{thm}[see Theorem \ref{thm:1.4} and Lemma \ref{lem:1.9}]\label{thm:I}\ \\
For a character $\chi : G \to \C^\times$ that is trivial on $G_\Sc$, the following are equivalent:
\begin{enumerate}[(i)]
\item $\chi |_T$ has depth zero, for every maximal torus $T \subset G$,
\item $\chi |_T$ has depth zero, for one maximal $F$-torus $\mc T \subset \mc G$
which contains a maximal unramified $F$-torus of $\mc G$,
\item $\ker (\chi)$ contains $G_{\ff,0+}$, for every parahoric subgroup 
$G_{\ff,0} \subset G$,
\item $\ker (\chi)$ contains $G_{\ff,0+}$, for one Iwahori subgroup $G_{\ff,0} \subset G$.
\end{enumerate}
\end{thm}
As in Proposition \ref{intro:prop:1.7}, we emphasize that Theorem \ref{thm:I} holds for any
connected reductive group over a non-archimedean local field, without any assumptions
on splitness or ramification. Of course
the wildly ramified groups present the biggest challenge in our proofs.

The results in this paper will be applied in our work on a local Langlands correspondence
for $G$-representations whose cuspidal supports are non-singular. 
First we do that in depth zero \cite{SoXu1}, then for representations which are tensor products 
of depth-zero representations with characters, and finally in arbitrary depths \cite{SoXu2}.

\section{Proofs of properties of the LLC for tori}

\begin{prop}\label{prop:1.8}
Let $E/F$ be a separable extension of local fields and let $\mc T$ be an
arbitrary $F$-torus. The following diagram commutes:
\[\begin{tikzcd}
    \Hom (\mc T (E), \C^\times) \arrow[]{r}{\sim} & H^1 (\mb W_E ,T^\vee) \\
\Hom (\mc T (F), \C^\times) \arrow[]{r}{\sim}\arrow[]{u}[swap]{N^*_{E/F}} & 
H^1 (\mb W_F, T^\vee)\arrow[]{u}{\Res^{\mb W_F}_{\mb W_E}}
\end{tikzcd}\]
In other words, $\varphi_\chi |_{\mb W_E} = \varphi_{\chi \circ N_{E/F}}$
for all smooth characters $\chi : \mc T (F) \to \C^\times$.
\end{prop}
\begin{proof}
First we consider non-archimedean local fields. 
Let $K / F$ be a finite Galois extension such that $K \supset E$ and $\mc T$ splits
over $K$. The construction of the LLC for $\mc T (E)$ and $\mc T (F)$ in 
\cite[\S 7.7]{Yu} can be summarized with the following diagram:
\begin{equation}\label{eq:1.19}
\! \adjustbox{scale=0.85}{\begin{tikzcd}
    \Hom (\mc T(E),\C^\times) \arrow[]{r}{} & \Hom (\mc T (K), \C^\times)_{\mr{Gal}(K/E)}
\arrow[]{r}{} & H^1 (\mb W_K, T^\vee)_{\mb W_E / \mb W_K} \arrow[]{r}{} &
H^1 (\mb W_E, T^\vee) \\
\Hom (\mc T(F),\C^\times) \arrow[]{r}{} \arrow[]{u}{N^*_{E/F}} & \Hom (\mc T (K),
\C^\times)_{\mr{Gal}(K/F)}\arrow[]{r}{} \arrow[dotted]{u}{(i)} & 
H^1 (\mb W_K, T^\vee)_{\mb W_F / \mb W_K} \arrow[]{r}{} \arrow[dotted]{u}{(ii)} &
H^1 (\mb W_F, T^\vee) \arrow[dotted]{u}{(iii)}
\end{tikzcd}}
\end{equation}
Here all the horizontal arrows are isomorphisms and the dotted arrows are defined 
such that the diagram commutes. The upper-left horizontal arrow comes from
$\mc T (E) = \mc T(K)^{\mr{Gal}(K/E)}$, the middle horizontal arrows come from local 
class field theory, and the right horizontal arrows are induced by corestriction maps.

We now deduce formulas for the dotted arrows. The map (i) is best seen as
\begin{align*}
\Hom (\mc T (K), \C^\times)_{\mr{Gal}(K/E)} &
\to 
\Hom (\mc T (K), \C^\times)_{\mr{Gal}(K/E)},\\
\chi & 
\mapsto 
\prod\limits_{\gamma \in \mr{Gal}(K/F) / \mr{Gal}(K/E)} \chi \circ \gamma^{-1}.
\end{align*}
This map is well-defined because we start with Gal$(K/E)$-coinvariants, and it factors
through $\Hom (\mc T (K), \C^\times)_{\mr{Gal}(K/F)}$. Let $\bar \gamma \in \mb W_F$ be 
any representative for $\gamma \in \mb W_F / \mb W_E$.
Similar to (i), the arrow (ii) in \eqref{eq:1.19} can be realized as the map
\begin{align}\label{eq:1.17}
\begin{split}
H^1 (\mb W_K ,T^\vee)_{\mb W_E / \mb W_K} & 
\to 
H^1 (\mb W_K, T^\vee)_{\mb W_E / \mb W_K}, \\
z & 
\mapsto 
\big[ w \mapsto \prod\nolimits_{\gamma \in \mb W_F / \mb W_E}
\bar \gamma \cdot z (\bar{\gamma}^{-1} w \bar \gamma) \big],
\end{split}
\end{align}
which is well-defined and factors through
$H^1 (\mb W_K ,T^\vee)_{\mb W_F / \mb W_K}$. We now study the right-hand square 
in the diagram \eqref{eq:1.19}, with (ii) replaced by \eqref{eq:1.17}:
\[\begin{tikzcd}
    H^1 (\mb W_K, T^\vee)_{\mb W_E / \mb W_K} \arrow[]{rr}{\mr{cor}_{K/E}} & & 
H^1 (\mb W_E, T^\vee) \\
H^1 (\mb W_K, T^\vee)_{\mb W_E / \mb W_K} \arrow[]{rr}{\mr{cor}_{K/F}} 
\arrow[]{u}{\eqref{eq:1.17}} & & H^1 (\mb W_F, T^\vee) \arrow[dotted]{u}{(iii)}
\end{tikzcd}\]
An explicit formula for the corestriction map 
$\mr{cor}_{K/F} : H^1 (\mb W_K,T^\vee) \to H^1 (\mb W_F, T^\vee)$ 
is given in \cite[\S 1.5]{NSW}. Let $\bar \gamma \in \mb W_F$ be a representative
for $\gamma \in \mb W_F / \mb W_K$, and let $z \in H^1 (\mb W_K, T^\vee)$, then
for $w \in \mb W_F$, we have
\begin{equation}\label{eq:1.18}
\begin{aligned}
\mr{cor}_{K/F}(z) (w) 
= \prod_{\substack{\gamma_1, \gamma_2 \in \mb W_F / \mb W_K : \\ 
w \in \gamma_1 \mb W_K \gamma_2^{-1}}} \bar{\gamma_1} \cdot z (\bar{\gamma_1}^{-1} w
\bar{\gamma_2})  
= \prod_{\substack{\gamma_3, \gamma_4 \in \mb W_F / \mb W_K : \\ 
w \gamma_4 \in \mb W_K}} \bar{\gamma_3} \cdot z (\bar{\gamma_3}^{-1} w
\overline{\gamma_4 \gamma_3}) .
\end{aligned}
\end{equation}
The third expression in \eqref{eq:1.18} is equal to the second because $\mb W_K$ 
is normal in $\mb W_F$. Note that $\gamma_4$ is uniquely determined by $w$, while
there is no condition on $\gamma_3 \in \mb W_F / \mb W_K$.
For $z \in H^1 (\mb W_K, T^\vee)$, we compute
\[
\mr{cor}_{K/E} \circ \eqref{eq:1.17} (z) (w) = 
\prod_{\substack{\gamma_5,\gamma_6 \in \mb W_E / \mb W_K \\ 
w \gamma_6 \in \mb W_K}}
\prod_{\gamma \in \mb W_F / \mb W_E} \bar{\gamma_5} \bar \gamma z \cdot
(\bar \gamma^{-1} \bar{\gamma_5}^{-1} w \overline{\gamma_6 \gamma_5} \bar \gamma) .
\]
We can use the elements $\bar{\gamma_5} \bar \gamma$ as representatives for
$\mb W_F / \mb W_K$, and then we see that
\[
\mr{cor}_{K/E} \circ \eqref{eq:1.17} (z)
= \Res^{\mb W_F}_{\mb W_E} \circ \mr{cor}_{K/F} (z) .
\]
Hence the diagram \eqref{eq:1.19} commutes, with 
$\Res^{\mb W_F}_{\mb W_E}$ as the map (iii). 

Now we consider archimedean local fields. There is only one nontrivial field
extension, namely $\C / \R$. We use the notations from \cite[\S 9.4]{Bor}.
Recall that $\mb W_\C = \C^\times$, and that 
$\mb W_\R = \C^\times \cup \tau \C^\times$ 
with 
$\tau^2 = -1$ 
and  
$\tau z \tau^{-1} = \bar z$ 
for  
$z \in \C^\times$.
We denote the action of the nontrivial element of $\mathrm{Gal}(\C/\R)$ on $X_* (\mc T)$
and on $X^* (\mc T)$ by $\sigma$. If the standard complex conjugation on
$X_* (\mc T) \otimes_\Z \C$ is written as $x \mapsto \bar x$, then the action of $\mathrm{Gal}(\C/\R)$ on $\mc T (\C)$ is given by $x \mapsto \sigma (\bar x)$ on $\mathrm{Lie}(\mc T(\C)) = X_* (\mc T) \otimes_\Z \C$. 

Let $\chi \in \Hom (\mc T (\R),\C^\times)$ have Langlands parameter $\varphi_\chi \in 
H^1 (\mb W_\R,T^\vee)$. There are $\mu,\nu \in X^* (\mc T) \otimes_\Z \C$ such that 
$\mu - \nu \in X^* (\mc T)$, $\nu = \sigma (\mu)$ and 
\begin{equation}\label{eq:1.20}
\varphi_\chi (z) = z^\mu \, \bar z^\nu \text{ for } z \in \C^\times .
\end{equation}
Replacing $\varphi_\chi$ by a $T^\vee$-conjugate, we may assume that $\varphi_\chi (\tau)$
is fixed by $\sigma$. Pick $h \in X^* (\mc T) \otimes_\Z \C$ such that 
$\varphi_\chi (\tau) = \exp (2\pi i h)$. By \cite[p. 117]{Lan1}, for $x \in X_* (\mc T)
\otimes_\Z \R$ with $\exp (x) \in \mc T (\R)$ we have\footnote{The corresponding formula in 
\cite[\S 9.4]{Bor} contains a typo, one $x$ should be $\bar x$.}
\[
\chi (\exp (x)) = \exp \big( \langle h, x - \sigma (\bar x) \rangle +
\langle \mu / 2, x + \sigma (\bar x) \rangle \big).
\]
For any $y \in X_* (\mc T) \otimes_\Z \C$, we compute
\begin{align*}
\chi \circ N_{\C/\R} (\exp (y)) & = \chi \big( \exp (y + \sigma (\bar y)) \big) \\
& = \exp \big( \langle h, y - \sigma (\bar y) + \sigma (\bar y) - y \rangle +
\langle \mu/2, y + \sigma (\bar y) + \sigma (\bar y) + y \rangle \big) \\
& = \exp \big( \langle \mu/2, 2y \rangle + \langle \mu/2, \sigma (2 \bar y) \rangle \big) \\
& = \exp \big( \langle \mu, y \rangle + \langle \nu, \bar y \rangle \big).
\end{align*}
Comparing the last expression with \eqref{eq:1.20} and \cite[\S 9.1]{Bor}, we see 
that $\phi_{\chi \circ N_{\C / \R}}$ is equal to $\phi_\chi |_{\mb W_\C}$.
\end{proof}

\begin{prop}\label{prop:1.7}
Let $\mc T$ be a (not necessarily tamely ramified) torus over a
non-archimedean local field $F$. The local Langlands correspondence for tori restricts to 
an isomorphism between the depth-zero parts on both sides of the correspondence:
\[
\Hom (T / T_{0+}, \C^\times) \cong H^1 \big( \mb W_F / \mb P_F, T^{\vee,\mb P_F} \big).
\]
\end{prop}
\begin{proof}
The right-hand side of \eqref{eq:1.15} is the kernel of the restriction map 
$H^1 (\mb W_F, T^\vee) \to H^1 (\mb I_F, T^\vee)$. Hence \eqref{eq:1.15} and the LLC for tori \eqref{eqn:LLC-for-tori}
induce the following isomorphisms of topological groups
\begin{equation}\label{eq:1.16}
\Hom (T^0 ,\C^\times) \cong \frac{\Hom (T,\C^\times)}{\Hom (T / T^0, \C^\times)}
\cong \frac{H^1 (\mb W_F, T^\vee)}{H^1 (\mb W_F / \mb I_F, T^{\vee,\mb I_F})} \cong
H^1_e (\mb I_F, T^\vee) .
\end{equation}
Here and below, the subscript e indicates that we only consider those classes that
can be extended to $H^1 (\mb W_F,T^\vee)$.  

Since $T^0$ is commutative, its pro-$p$ radical $T_{0+}$ is the unique maximal 
pro-$p$ subgroup of $T^0$. By Pontryagin duality, $\Hom (T_{0+},\C^\times)$
is the unique maximal pro-$p$ quotient of $\Hom (T^0,\C^\times)$.~In other words,
$\Hom (T^0 / T_{0+},\C^\times)$ is the largest subgroup of $\Hom (T^0 , \C^\times)$
that has no nontrivial pro-$p$ quotient groups. 

Recall from \cite[Proposition 7.5.2]{NSW} that $\mb I_F / \mb P_F$ is isomorphic to 
the quotient of the profinite completion of $\Z$ by the $p$-adic integers. This group 
is profinite and has no nontrivial
pro-$p$ subgroups, so every element of $H^1 (\mb I_F / \mb P_F, T^{\vee,\mb P_F})$
has order coprime to $p$. Therefore $H^1 \big( \mb I_F / \mb P_F, T^{\vee,\mb P_F} \big)$
has no nontrivial pro-$p$ quotient groups. Consider an element $x \in H^1 (\mb I_F,T^\vee)
\setminus H^1 \big( \mb I_F / \mb P_F, T^{\vee,\mb P_F} \big)$. By the continuity of $x$,
there exists an open subgroup $\mb K_F \subset \mb P_F$ that fixes $T^\vee$ pointwise, 
such that $x |_{\mb K_F} = 1$. Then $\mb P_F / \mb K_F$ is a finite $p$-group and
the image of $x$ in $H^1 (\mb P_F / \mb K_F ,T^\vee)$ is nontrivial, so its order is
a power of $p$. Thus the group generated by $H^1 \big( \mb I_F / \mb P_F, 
T^{\vee,\mb P_F} \big) \cup \{x\}$ does have a nontrivial pro-$p$ quotient group.

We conclude that the isomorphism \eqref{eq:1.16} sends $\Hom (T^0 / T_{0+},\C^\times)$
bijectively to the largest subgroup of $H^1_e (\mb I_F, T^\vee)$ that has no nontrivial
pro-$p$ quotients, namely $H^1_e \big( \mb I_F / \mb P_F, T^{\vee,\mb P_F} \big)$. Now we take
the images of these groups in the two middle terms of \eqref{eq:1.16}, which produces
an isomorphism
\[
\frac{\Hom (T / T_{0+},\C^\times)}{\Hom (T / T^0, \C^\times)} \cong \frac{H^1 (\mb W_F 
/ \mb P_F, T^{\vee,\mb P_F})}{H^1 (\mb W_F / \mb I_F, T^{\vee,\mb I_F})} .
\]
To conclude, we combine this with \eqref{eq:1.15}.
\end{proof}

\section{Depth-zero characters of reductive $p$-adic groups}

Let $\mc G$ be a connected reductive group defined over a local field $F$.
We are interested in the smooth characters of $G = \mc G (F)$.
Let $\mc G_\Sc$ be the simply connected cover of the derived subgroup $\mc G_\der$ of $\mc G$. 
The quotient map $\mc G_\Sc \to \mc G_\der$ sends $G_\Sc$ to a normal subgroup of $G$, which is in general smaller than $G_\der$. The cokernel of the canonical map $G_\Sc \to G$
will be abbreviated as $G / G_\Sc$ (even though $G_\Sc$ is usually not a subgroup of $G$).
For most reductive groups, $G_\Sc$ is generated by commutators, so it automatically
lies in the kernel of any character. This may fail, however, when $G_\der$ is anisotropic, which 
is why we explicitly require that our characters of $G$ are trivial on the image of $G_\Sc$.

In \cite[p.~124--125]{Lan1}, Langlands constructed a natural map from
$H^1 (\mb W_F,Z(G^\vee))$ to smooth characters of $G / G_\Sc$.
Since we could not find a proof of the bijectivity of this map in the 
literature, we provide it below.

\begin{thm}\label{thm:1.1}
Let $\mc G$ be a connected reductive group over a local field $F$.
There exists a natural isomorphism of topological groups
\[
H^1 (\mb W_F,Z(G^\vee)) \isom \Hom (G / G_\Sc,\C^\times) :\; \varphi \mapsto \chi_\varphi
\]
\end{thm}
\begin{proof}
Langlands' construction of this natural homomorphism is presented in \cite[\S 10.2]{Bor},
which we follow. One chooses an $F$-torus $\mc D$ 
satisfying $H^1 (F,\mc D) = \{1\}$ and an extension of $F$-groups
\begin{equation}\label{eq:1.2}
1 \to \mc D \to \tilde{\mc G} \to \mc G \to 1,     
\end{equation}
such that $\mc G_\Sc \to \mc G$ lifts to an embedding $\mc G_\Sc \to \tilde{\mc G}$. 
By \eqref{eq:1.2} and the condition $H^1 (\mb W_F,\mc D) = \{1\}$, we have a group 
isomorphism $\mc G (F) \cong \tilde{\mc G} (F) / \mc D (F)$. Hence
\begin{equation}\label{eq:1.3}
\Hom (G / G_\Sc,\C^\times) =
\big\{ \chi \in \Hom (\tilde{\mc G} (F) / \mc G_\Sc (F),\C^\times) : 
\mc D (F) \subset \ker \chi \big\} .
\end{equation}
The group $\tilde{\mc G} / \mc G_\Sc$ is an $F$-torus. By \cite[Proposition 7]{BrTi1}
(for any non-archimedean $F$) and by \cite[\S 6]{Kot} (for any $F$ of characteristic 
zero), $H^1 (F,\mc G_\Sc)$ is trivial. Hence there is a short exact sequence
\[
1 \to \mc G_\Sc (F) \to \tilde{\mc G}(F) \to (\tilde{\mc G}/\mc G_\Sc)(F) \to 1 .
\]
Thus the LLC for tori (Theorem \ref{thm:1.6}) provides natural isomorphisms 
\begin{equation}\label{eq:1.4}
\Hom ( \tilde{\mc G} (F) / \mc G_\Sc (F), \C^\times) \cong 
\Hom ( (\tilde{\mc G}/\mc G_\Sc)(F), \C^\times) \cong 
H^1 ( \mb W_F,  (\tilde{\mc G}/\mc G_\Sc)^\vee ) .
\end{equation}
By \eqref{eq:1.4} and functoriality in Theorem \ref{thm:1.6}, \eqref{eq:1.3} is
naturally isomorphic to
\begin{equation}\label{eq:1.5}
\ker \big( H^1 ( \mb W_F,  (\tilde{\mc G}/\mc G_\Sc)^\vee ) \to
H^1 (\mb W_F, \mc D^\vee) \big) .
\end{equation}
Now the map in the statement is obtained from the natural map from $H^1 (\mb W_F, Z(G^\vee))$ to
\eqref{eq:1.5} followed by \eqref{eq:1.3}. Since the derived group of $\tilde{\mc G}$
is simply connected, $Z(\tilde{\mc G}^\vee)$ is connected and is equal to 
$(\tilde{\mc G} / \mc G_\Sc)^\vee$. As noted in \cite[10.2.(2)]{Bor}, there is 
a short exact sequence  of $\mb W_F$-groups 
\[
1 \to Z(\mc G^\vee) \to Z(\tilde{\mc G}^\vee) \to \mc D^\vee \to 1 , 
\]
which gives rise to a long exact sequence in group cohomology
\begin{equation}\label{eq:1.6}
\begin{array}{ccccccc}
Z(\mc G^\vee)^{\mb W_F} & \to & Z(\tilde{\mc G}^\vee)^{\mb W_F} & \to & (\mc D^\vee)^{\mb W_F} & \to \\
H^1 (\mb W_F, Z(\mc G^\vee)) & \to & H^1 (\mb W_F, Z(\tilde{\mc G}^\vee)) & \to & 
H^1 (\mb W_F, \mc D^\vee) .
\end{array}
\end{equation}
The Kottwitz isomorphism (see \cite[\S 6]{Kot} and \cite[Theorem 2.1]{Tha}) shows that
\[
\pi_0 \big( (\mc D^\vee)^{\mb W_F} \big) \cong \Hom (H^1 (F,\mc D), \C^\times) = \{1\} ,
\]
thus $(\mc D^\vee)^{\mb W_F}$ is connected.~Then $(\mc D^\vee)^{\mb W_F}$ is equal to $\exp \big( \mr{Lie}(\mc D^\vee)^{\mb W_F} \big)$, onto which $\exp \big( \mr{Lie} (Z(\tilde{\mc G}^\vee))^{\mb W_F} \big)$ surjects.~Hence the map $Z(\tilde{\mc G}^\vee)^{\mb W_F} \to (\mc D^\vee)^{\mb W_F}$ in \eqref{eq:1.6} is surjective.~By the exactness of \eqref{eq:1.6}, the natural map from
$H^1 (\mb W_F, Z(\mc G^\vee))$ to \eqref{eq:1.5} is bijective. Since all the maps in the construction are continuous, so is the final isomorphism (in both directions).
\end{proof}

From now on $F$ is a non-archimedean local field, and our main object of study is the group of characters 
\[
\Xo (G) := \{ \chi : G / G_\Sc \to \C^\times \mid \chi |_T \text{ has depth
zero for all maximal tori } T \subset G \} .
\]
We define $\Xo (G^\vee) \subset H^1 (\mb W_F, Z(G^\vee)$ as the image of $\Xo (G)$ 
under Theorem \ref{thm:1.1}. Let us present two more explicit descriptions of $\Xo (G^\vee)$.

\begin{lem}\label{lem:1.2}
\enuma{
\item $\Xo (G^\vee)$ is equal to the set of $\varphi \in H^1 (\mb W_F, Z(G^\vee))$ such that, 
for all maximal tori $T \subset G$, $\varphi$ has depth zero in $H^1 (\mb W_F ,T^\vee)$.
\item $\Xo (G^\vee)$ contains $H^1 \big( \mb W_F / \mb P_F, Z(G^\vee)^{\mb P_F} \big)$.
\item If $\mc G$ splits over a tamely ramified extension of $F$, then
\[
\Xo (G^\vee) = H^1 \big( \mb W_F / \mb P_F, Z(G^\vee)^{\mb P_F} \big).
\]
}
\end{lem}
\begin{proof}
(a) From any L-embedding ${}^L T \to {}^L G$, we obtain a homomorphism
\begin{equation}\label{eq:1.7}
H^1 (\mb W_F, Z(G^\vee)) \to H^1 (\mb W_F, T^\vee) .    
\end{equation}
It does not depend on the choice of the $L$-embedding because only central
elements of $G^\vee$ play a role in the domain. For
$\varphi \in H^1 (\mb W_F, Z(G^\vee))$, by the naturality of
Langlands' map in Theorem \ref{thm:1.1}, the following coincide:
\begin{itemize}
\item the character of $T$ determined by $\varphi$ via the map \eqref{eq:1.7} followed by 
the LLC for tori \eqref{eqn:LLC-for-tori},
\item the restriction to $T$ of the character $\chi_\varphi$
from \eqref{eq:1.1}.
\end{itemize}
By Proposition \ref{prop:1.7}, the LLC for tori preserves the depth-zero property. Thus, 
for any maximal torus $T \subset G$, the condition that $\chi_\varphi |_T$ has 
depth zero is equivalent to the condition that $\varphi$ as an element of 
$H^1 (\mb W_F ,T^\vee)$ has depth zero.

(b) If $\varphi \in H^1 \big( \mb W_F / \mb P_F, Z(G^\vee)^{\mb P_F} \big)$, then clearly
$\varphi$ has depth zero as an element of $H^1 (\mb W_F, T^\vee)$, for any maximal torus
$T \subset G$. Combine this with part (a).

(c) By assumption, there exists a tamely ramified maximal $F$-torus $\mc T \subset \mc G$. 
Let $\varphi \in \Xo (G^\vee)$, thus $\chi_\varphi \in \Xo (G)$ and in particular 
$\chi_\varphi |_T$ has depth zero. Reversing the argument for part (a), $\varphi$ as an 
element of $H^1 (\mb W_F, T^\vee)$ has depth zero. Now $\varphi$ is an equivalence class for 
$T^\vee$-conjugacy, so there exists an element $t \in T^\vee$ such that $\varphi (w) = 
t w( t^{-1})$ for all $w \in \mb P_F$. Since $\mb P_F$ fixes $T^\vee$ pointwise (by the 
tame ramification of $\mc T$), we have $\varphi (w) = 1$ for all $w \in \mb P_F$.
This implies that $\varphi \in H^1 \big( \mb W_F / \mb P_F ,Z(G^\vee)^{\mb P_F} \big)$. 
\end{proof}

For an alternative characterization of $\Xo (G)$, we use some Bruhat--Tits theory. 
Let $F_\nr$ be a maximal unramified extension of $F$. Let $\mc S$ be a maximal 
$F$-split torus of $\mc G$, and let $\mc S'$ be a maximal 
$F_\nr$-split $F$-torus of $\mc G$ which contains $\mc S$ (for its existence, see for 
example \cite[Proposition 9.3.4]{KaPr}). By well-known results of Lang and Steinberg
(see for example \cite[Example 2.3.2 and Theorem 2.3.3]{KaPr}), the group $\mc G$ is 
quasi-split over $F_\nr$. Then $\mc T' := Z_{\mc G}(\mc S')$ is a minimal 
$F_\nr$-Levi subgroup and a maximal $F$-torus of $\mc G$. 

Let $\mc B (\mc G,F)$ be the Bruhat--Tits building of $F$. For every facet
$\ff$ of $\mc B (\mc G,F)$, we have the stabilizer group $G_\ff$ and its Moy--Prasad subgroups 
$G_{\ff,r}$ for $r \in \R_{\geq 0}$ \cite{MoPr2}.~Similarly, we have the
filtration subgroups $T'_r$ of the unique parahoric subgroup of $T'$.

\begin{prop}\label{prop:1.3}
For every facet $\ff$ of $\mc B (\mc G,F)$,
the group $T'_{0+} \cdot \mr{Im}(G_\Sc \to G)$ contains $G_{\ff,0+}$.
\end{prop}
\begin{proof}
Recall that $\mc B (\mc G_\Sc,F)$ is the reduced building of $G$, and that the set of its 
apartments is naturally in bijection with the set of apartments of $\mc B (\mc G,F)$.
By \cite[Proposition 9.3.22]{KaPr}, $G_\Sc$ acts transitively on the set of apartments
of $\mc B (\mc G,F)$. Since $G_{g \ff,0+} = g G_{\ff, 0+} g^{-1}$ for $g \in G_\Sc$,
we may assume that $\ff$ lies in the apartment $\mh A_S$ associated to $S$.

The Moy--Prasad subgroups $G_{\ff,r}$ give rise to $\mf o_F$-Lie subalgebras 
$\mf g_{\ff,r}$ of $\mf g = \mr{Lie}(G)$. 
Let $q : \mc G_\Sc \to \mc G_\der$ be the quotient map and consider the map
\begin{equation}\label{eq:1.9}
T' \times G_\Sc \to G : (t,g) \mapsto t \, q(g) .    
\end{equation}
By the construction of $G_{\ff,r}$ in \cite[Definition 7.3.3 and \S 9.8]{KaPr}, 
for any $r \in \R_{\geq 0}$, \eqref{eq:1.9} restricts to a map
\begin{equation}\label{eq:1.10}
T'_r \times G_{\Sc,\ff,r} \to G_{\ff,r} .
\end{equation}
The Lie algebra of $T' \times G_\Sc$
is $\mf t' \oplus \mf g_\der$, and similar to \eqref{eq:1.10}, the addition map
\begin{equation}\label{eq:1.8}
\mf t' \oplus \mf g_\der \to \mf g \quad \text{sends} \quad
\mf t'_r \oplus \mf g_{\der,\ff,r} \;\text{to}\; \mf g_{\ff,r}. 
\end{equation}
Let $0 < r < s \leq 2r \in \R$, and assume that either $r \leq 1$ or $s \leq 2r - 1$. 
The maps \eqref{eq:1.10} and \eqref{eq:1.8} fit in a commutative diagram 
with the Moy--Prasad isomorphism (see \cite[Theorem 13.5.1]{KaPr}): 
\begin{equation}\label{eq:1.11}
\begin{tikzcd}
G_{\ff,r} / G_{\ff,s}  \arrow[]{r}{\sim} & \mf g_{\ff,r} / \mf g_{\ff,s} \\
(T'_r \times G_{\Sc,\ff,r}) / (T'_s \times G_{\Sc,\ff,s})  \arrow[]{r}{\sim}\arrow[]{u}{} &
(\mf t'_r \oplus \mf g_{\der,\ff,r}) / (\mf t'_s \oplus \mf g_{\der,\ff,s})\arrow[]{u}{} 
\end{tikzcd}
\end{equation}
Although \eqref{eq:1.9} is a group homomorphism only when $\mc G$ is a torus,
the commutativity of $G_{\ff,r} / G_{\ff,s}$ implies that both vertical maps in 
\eqref{eq:1.11} are group homomorphisms. We claim that the 
vertical map on the right-hand side in \eqref{eq:1.11} is surjective.

In the next steps, we consider Lie algebras as $F$-schemes.
We follow the convention that $\mf g (F)$ is the same as $\mf g$,
and we let $\mf g (F_\nr) := \mf g \otimes_F F_\nr$ be the Lie algebra of $F_\nr$-rational 
points of $\mf g$. 
It carries an action of $\mb W_F / \mb I_F = \langle \Fr_F \rangle$, and by
definition $\mf g_{\ff,r} = \mf g (F_\nr)_{\ff,r}^{\; \Fr_F}$. The same
holds for $\mf t'_r$ and $\mf g_{\der,\ff,r}$. Therefore, for our claim, we need to
check that the addition map
\begin{equation}\label{eq:1.12}
\mf t' (F_\nr)_r \oplus \mf g_\der (F_\nr)_{\ff,r} \to \mf g (F_\nr)_{\ff,r}    
\end{equation}
is surjective.~Our earlier assumption on $\ff$ implies that it lies in the apartment of
$\mc B (\mc G,F_\nr)$ associated to $\mc S'$. By \cite[Definition 7.3.3]{KaPr},
$\mf g (F_\nr)_{\ff,r}$ is generated by $\mf t' (F_\nr)_r$ and subspaces
$\mf u_\alpha (F_\nr)_{\ff,r}$, where $\mf u_\alpha (F_\nr) \subset \mf g (F_\nr)$ 
is the weight space for a root $\alpha \in R(\mc G, \mc S')$.
These $\mf u_\alpha (F_\nr)_{\ff,r}$ are also contained in $\mf g_\der (F_\nr)_{\ff,r}$,
which shows the surjectivity of \eqref{eq:1.12}. It follows that the vertical map 
on the right-hand side in the diagram \eqref{eq:1.11} is surjective.

By the commutativity of \eqref{eq:1.11}, the vertical map on the left-hand side
\begin{equation}\label{eq:1.13}
(T'_r \times G_{\Sc,\ff,r}) / (T'_s \times G_{\Sc,\ff,s}) \to
G_{\ff,r} / G_{\ff,s} 
\end{equation}
is surjective as well. The jumps in the Moy--Prasad filtration of
$G_\ff$ form a discrete subset of $\R_{\geq 0}$, so we can find $r_0 > 0$
such that $G_{\ff,0+} = G_{\ff,r_0}$. Extend this to a sequence
$r_0 < r_1 < r_2 < \ldots$ with $\lim_{n \to \infty} r_n = \infty$, such that
each pair $(r_n, r_{n+1})$ satisfies the conditions in \cite[Theorem 13.5.1]{KaPr}. 
This provides a filtration of $G_{\ff,0+}$ with 
successive subquotients as in the top row of \eqref{eq:1.11}. The same holds
for $T'$ and $G_\Sc$, with the same $r_n$. Now the surjectivity of
\eqref{eq:1.13} with $r = r_n, s = r_{n+1}$ and the completeness of
$\mf g$ with respect to the $p$-adic topology imply that \eqref{eq:1.10} is
surjective for $r = r_0$. In other words, $T'_{0+} \times 
\mr{Im}(G_{\Sc,\ff,0+} \to G_{\ff,0+})$ surjects onto $G_{\ff,0+}$.
\end{proof}

Recall that a chamber $C$ of $\mc B (\mc G,F)$ is a facet of maximal dimension, 
and that the associated parahoric subgroup $G_{C,0}$ is called an Iwahori 
subgroup of $G$. While $G_{C,0}$ is minimal among the parahoric subgroups of
$G$, its pro-unipotent radical $G_{C,0+}$ is maximal among the subgroups 
$G_{\ff,0+}$ for facets $\ff$ of $\mc B (\mc G,F)$ \cite[Lemma 7.4.12]{KaPr}. \\
We are ready to establish an alternative description of $\Xo (G)$. It shows
that tensoring by elements of $\Xo (G)$ stabilizes the category of depth-zero
representations of $G$.

\begin{thm}\label{thm:1.4}
\enuma{
\item The group $\Xo (G)$ is equal to
\begin{equation*}\label{eqn:thm:1.4}
\{ \chi : G / G_\Sc \to \C^\times \mid  G_{\ff,0+} \subset
\ker (\chi) \text{ for every facet } \ff \text{ of } \mc B (\mc G,F) \}.
\end{equation*}
\item $\Xo (G) = \{ \chi : G / G_\Sc \to \C^\times \mid G_{C,0+} \subset
\ker (\chi) \text{ for one chamber } C \text{ of } \mc B (\mc G,F) \}$. 
}
\end{thm}
\begin{proof}
(a) For $\chi_\varphi \in \Xo (G)$, $\chi_\varphi |_{T'}$ has depth zero, thus 
$T'_{0+} \subset \ker (\chi_\varphi)$. By Proposition \ref{prop:1.3}, 
$\ker (\chi_\varphi)$ contains $G_{\ff,0+}$ for all $\ff$.

Consider $\chi$ as in the statement, and let $T \subset G$ be a maximal
torus.~The group $T^0$ is compact, so by the Bruhat--Tits fixed point
theorem $T^0$ stabilizes a facet $\ff$ of $\mc B (\mc G,F)$. Recall
that by definition $T_{0+}$ is the maximal pro-$p$ subgroup of $T^0$.
Let $G^0$ be the kernel of the Kottwitz homomorphism for $G$. By the
functoriality of the Kottwitz homomorphism \cite[\S 11.5]{KaPr} and
by \cite[Proposition 7.7.5]{KaPr}, we have 
\[
T^0 \subset G^0 \cap G_\ff = G_{\ff,0}.
\]
By \cite[Proposition 13.5.2.1]{KaPr}, we know that $G_{\ff,0+}$ is a pro-$p$
subgroup of $G_{\ff,0}$, but it is not necessarily maximal. Together with 
\cite[Proposition 13.5.2.2]{KaPr}, we see that the maximal pro-$p$ subgroups
of $G_{\ff,0}$ are the groups $G_{\ff',0+}$, with $\ff'$ a chamber of 
$\mc B (\mc G,F)$ such that $G_{\ff',0+} \subset G_{\ff,0}$.\footnote{The condition
on $\ff'$ is equivalent to $\ff \subset \overline{\ff'}$, but we do not need
this.} Hence $T_{0+} \subset G_{\ff',0+}$ for such a chamber $\ff'$. By the 
assumption on $\chi$, this means $T_{0+} \subset \ker (\chi)$. 

This holds for any maximal torus $T \subset G$, thus $\chi \in \Xo (G)$.\\
(b) By part (a), $\Xo (G)$ is contained in the right hand side. Let $\chi$ be
as in the statement and let $\ff$ be an arbitrary facet of $\mc B (\mc G,F)$.
Let $C'$ be a chamber of $\mc B (\mc G,F)$ such that $\overline{C'} \supset \ff$,
so $G_{\ff,0+} \subset G_{C',0+}$.~There is only one $G$-orbit of chambers in $\mc B (\mc G,F)$ \cite[Proposition
9.3.22]{KaPr}, so we can find $g \in G$ such that $g C = C'$. The 
equivariance properties of Moy--Prasad subgroups \cite[Proposition 13.2.5]{KaPr}
gives 
\[
\ker (\chi) = g \ker (\chi) g^{-1} \supset g G_{C,0+} g^{-1} = G_{g C,0+}
= G_{C',0+} \supset G_{\ff,0+} .
\]
Now part (a) shows that $\chi \in \Xo (G)$.
\end{proof}

The next lemma shows that the definition of $\Xo (G)$ can be simplified: instead 
of requiring $\chi|_T$ to have depth zero 
``for all maximal tori", it suffices to check on one well-chosen torus in $G$.

\begin{lem}\label{lem:1.9}
Let $\mc T$ be a maximal $F$-torus of $\mc G$ that contains a maximal unramified
$F$-torus $\mc T_\nr$ of $\mc G$. Suppose that $\varphi \in H^1 (\mb W_F, Z(G^\vee))$ and 
that $\chi_\varphi |_T$ has depth zero. Then $\varphi \in \Xo (G^\vee)$ and $\chi_\varphi \in
\Xo (G)$. In particular
\[
\Xo (G) = \{ \chi : G / G_\Sc \to \C^\times \mid \chi |_T \text{ has depth zero} \}.
\]
\end{lem}
\begin{proof}
Since $\mc G$ is quasi-split over $F_\nr$, $Z_{\mc G}(\mc T_\nr)$ is a maximal torus
of $\mc G$ and hence $Z_{\mc G}(\mc T_\nr) = \mc T$. Let $E/F$ be a finite unramified
extension splitting $\mc T_\nr$. Let 
\[
\chi_{\varphi,E} : \mc G(E) / \mc G_\Sc (E) \to \C^\times
\]
be the character associated to $\varphi |_{\mb W_E}$ in Theorem \ref{thm:1.1}. By the
compatibility of Theorem \ref{thm:1.1} with the LLC for tori (see the proof of
Lemma \ref{lem:1.2}(a)), $\chi_{\varphi,E} |_{\mc T (E)}$ is the character with Langlands
parameter $\varphi |_{\mb W_E} \in H^1 (\mb W_E, T^\vee)$. By Proposition \ref{prop:1.8}, 
this character is equal to $\chi_\varphi |_{\mc T (F)} \circ N_{E/F}$.

The map $N_{E/F} : \mc T(E) \to \mc T(F)$ can be viewed as a homomorphism of $F$-tori.
By the functoriality of filtrations of tori \cite[Definition 7.2.2 and Proposition
B.10.10]{KaPr}, $N_{E/F} (\mc T (E)_{0+}) \subset \mc T (F)_{0+}$.
By assumption, $\chi_\varphi |_{\mc T (F)}$ has depth zero, so
\[
\mc T(E)_{0+} \subset \ker (\chi_\varphi |_{\mc T(F)} \circ N_{E/F}) =
\ker (\chi_{\varphi,E} |_{\mc T (E)}) .
\]
This implies 
\[
\mc T(E)_{0+} \cdot \mr{Im}(\mc G_\Sc (E) \to \mc G (E)) 
\;\subset\; \ker (\chi_{\varphi,E}). 
\]
The $E$-torus $\mc T$ contains the maximal unramified torus $\mc T_\nr$, which is also
a maximal $E$-split torus in $\mc G$. Thus $\mc T (E)$ has the form of $T'$ above.
Proposition \ref{prop:1.3} applies and we see that $\ker (\chi_{\varphi,E})$ contains
$\mc G (E)_{\ff,0+}$, for every facet $\ff$ of $\mc B (\mc G,E)$. 

By Theorem \ref{thm:1.4}, $\chi_{\varphi,E} |_{\mc D (E)}$ has depth zero for 
every maximal $E$-torus $\mc D \subset \mc G$, thus in particular for every maximal
$F$-torus. By Proposition \ref{prop:1.7}, $\varphi |_{\mb W_E} \in
H^1 (\mb W_E, \mc D^\vee)$ has depth zero. Since $E/F$ is unramified, $\mb P_E = \mb P_F$
and $\varphi \in H^1 (\mb W_F, \mc D^\vee)$ also has depth zero. By Lemma \ref{lem:1.2}(a),
$\varphi \in \Xo (G^\vee)$ and $\chi_\varphi \in \Xo (G)$.
\end{proof}

\end{document}